\documentclass[final,3p,times]{elsarticle}
\usepackage{amsmath,amssymb,amsthm,graphicx,float}

\newcommand{\dM}{b}
\newtheorem{thm}{Theorem}[section]
\newtheorem{lem}[thm]{Lemma}

\newtheorem{Def}[thm]{Definition}

\newtheorem{examp}[thm]{Example}

\begin{document}

\begin{frontmatter}

\title{Density Estimation on Manifolds with Boundary}

\author{Tyrus Berry\corref{cor1}}
\address{4400 University Drive, Fairfax, Virginia 22030}
\ead{tberry@gmu.edu}
\cortext[cor1]{Corresponding author}
\author{Timothy Sauer\corref{cor2}}
\address{4400 University Drive, Fairfax, Virginia 22030}
\ead{tsauer@gmu.edu}

\begin{abstract}
Density estimation is a crucial component of many machine learning methods, and manifold learning in particular, where geometry is to be constructed from data alone.   A significant practical limitation of the current density estimation literature is that methods have not been developed for manifolds with boundary, except in simple cases of linear manifolds where the location of the boundary is assumed to be known.  We overcome this limitation by developing a density estimation method for manifolds with boundary that does not require any prior knowledge of the location of the boundary.  To accomplish this we introduce statistics that provably estimate the distance and direction of the boundary, which allows us to apply a cut-and-normalize boundary correction.  By combining multiple cut-and-normalize estimators we introduce a consistent kernel density estimator that has uniform bias, at interior and boundary points, on manifolds with boundary.
\end{abstract}


\end{frontmatter}

\section{Introduction}\label{intro}

Nonparametric density estimation has become an important tool in statistics with a wide range of applications to machine learning, especially for high-dimensional data.  The increasing size and complexity of measured data creates the possibility of understanding increasingly complicated phenomena for which there may not be sufficient `first principles' understanding to enable effective parametric modeling.  The exponential relationship between model complexity (often quantified as dimensionality) and data requirements, colloquially known as the {\it curse of dimensionality}, demands that new and innovative priors be developed.  A particularly effective assumption is the \emph{geometric prior}, which assumes that the data lies on a manifold that is embedded in the ambient Euclidean space where the data is sampled.  The geometric prior is nonparametric in that it does not assume a particular manifold or parametric form, merely that the data is restricted to lying on \emph{some} manifold.  This prior allows us to separate the \emph{intrinsic} dimensionality of the manifold, which may be low, from the \emph{extrinsic} dimensionality of the ambient space, which is often high.  

Recently the geometric prior has received some attention in the density estimation field \cite{hendriks1990nonparametric,pelletier2005kernel,kim2013geometric,nips2009}, although use of these methods remains restricted for several reasons.  For example, the methods of \cite{hendriks1990nonparametric,pelletier2005kernel} require the structure of the manifold to be known a priori.  Recently, in the applied harmonic analysis literature a method known as \emph{diffusion maps} has been introduced which learns the structure of the manifold from the data \cite{BN,diffusion}.  These methods have also been extended to a large class of noncompact manifolds \cite{heinThesis,hein1,hein2} with natural assumptions on the geometry of the manifold.  The assumptions introduced in \cite{heinThesis} include all compact manifolds, as well as many non-compact manifolds, such as any linear manifold, which implies that standard kernel density estimation theory on Euclidean spaces is included as a special case.  While these manifold learning methods make implicit assumptions on the geometry of the underlying manifold (such as bounded curvature), kernel density estimation requires knowledge of the dimension of the manifold in order to obtain the correct normalization factor.  For ease of exposition, we will assume the dimension of the manifold is known, although this is not necessary: In Appendix B we include a practical method of empirically tuning the bandwidth parameter that also estimates the dimension.

The remaining significant limitation of applying existing manifold density estimators to real problems is the restriction to manifolds without boundary. One exception is the special case of subsets of the real line where the location of the boundary is assumed to be known. This case has been thoroughly studied, and consistent estimators have been developed \cite{boundary93,boundary96,boundary99,boundary00,boundary05,boundary14}.

Here we introduce a consistent kernel density estimator for manifolds with (unknown) boundary that has the same asymptotic bias in the interior as on the boundary. The first obstacle to such an estimator is that a conventional kernel does not integrate to one near the boundary. Therefore the normalization factor must be corrected in a way that is based on the distance to the boundary, which is not known {\it a priori}.

To locate the boundary, we couple the standard kernel density estimator (KDE) with a second calculation, a kernel weighted average of the vectors from every point in the data set to every other point, which we call the boundary direction estimator (BDE).  We present asymptotic analysis of the BDE that shows that if the base point is near a boundary, the negative of the resulting average vector will point toward the nearest point on the boundary.  We also use the asymptotic expansion of this vector to find a lower bound on the distance to the boundary.  Our new density estimate at this base point does not include the data which lie beyond the lower bound in the direction of the boundary.  This creates a virtual boundary in the tangent space which is simply a hyperplane (dimension one less than the manifold) at a known distance from the base point.  Creating a known virtual boundary allows us to bypass the above obstacle -- we can now renormalize the kernel so that it integrates exactly to one at each base point, similar to the cut-and-normalize kernels that are used when the boundary is {\it a priori} known.  For points in the interior (or for manifolds without boundary), the lower bound on the distance to the boundary goes to infinity in the limit of large data, and we recover the standard kernel density estimation formula.  Moreover, using standard methods of constructing higher order kernels, we find a formula for a kernel density estimate with the same asymptotic bias for interior points and points on the boundary.

In Section \ref{background} we briefly review nonparametric density estimation on embedded manifolds.  The boundary correction method using BDE is introduced in Section \ref{boundary}, and the results are demonstrated on several illustrative examples.  We conclude with a brief discussion in Section \ref{discussion}.

\section{Background}\label{background}

Assume one is given $N$ samples $\{X_i\}_{i=1}^N$ (often assumed to be independent) of a probability distribution on $R^n$  with a density function $f(x)$. The problem of nonparametric density estimation is to find an estimator $f_N(x)$ that approximates the true density function.   
A kernel density estimator $f_N$ is typically constructed  \cite{Parzen62,loftsgaarden65,whittle58} as
\begin{align}\label{ambientKDE} f_N(x) = \frac{1}{N h_N^n}\sum_{i=1}^N K\left(\frac{||x-X_i||}{h_N}\right) \end{align}
where the kernel function is defined via a univariate shape function $K$ and $h_N \to 0$ as $N\to \infty$.  The kernel function must be normalized to integrate to $1$ for each $N$ to have a consistent estimator.

%

The standard KDE formulation \eqref{ambientKDE} assumes that the density is supported on the Euclidean space from which the data is sampled. However, real data may be restricted to lie on a lower dimensional submanifold of this Euclidean space.  This assumption, which we call the \emph{geometric prior}, is a potential workaround to the curse of dimensionality for high dimensional data.  Since the geometric prior assumes that the density is supported on a submanifold of the ambient Euclidean space, we may assume that the intrinsic dimensionality is small even when the extrinsic dimensionality is large.  

Nonparametric density estimation on manifolds essentially began with Hendriks \cite{hendriks1990nonparametric}, who modernized the Fourier approach of \cite{whittle58} using a generalized Fourier analysis on compact Riemannian manifolds without boundary, based on the eigenfunctions of the Laplace-Beltrami operator.  The limitation of \cite{hendriks1990nonparametric} in practice is that it requires the eigenfunctions of the Laplace-Beltrami operator on the manifold to be known, which is equivalent to knowing the entire geometry.  A kernel-based method of density estimation was introduced in \cite{pelletier2005kernel}. In this case the kernel was based on the geodesic distance between points on the manifold, which is again equivalent to knowing the entire geometry.  More recently, a method which uses kernels defined on the tangent space of the manifold was introduced \cite{kim2013geometric}.  However, evaluating the kernel of \cite{kim2013geometric} requires lifting points on the manifold to the tangent space via the exponential map, which yet again is equivalent to knowing the geometry of the manifold.  (See, for example, \cite{laplacianBook} which shows that the Riemannian metric can be recovered from either the Laplace-Beltrami operator, the geodesic distance function, or the exponential map.)  The results of \cite{hendriks1990nonparametric,pelletier2005kernel,kim2013geometric}, in addition to being restricted to compact manifolds without boundary, are limited to manifolds which are known {\it a priori}, and cannot be applied to data lying on an unknown manifold embedded in Euclidean space.  

The insight of \cite{diffusion} was that as the bandwidth $h_N$ decreases and the kernel function approaches a delta function, the kernel is essentially zero outside a ball of radius $h_N$.  Inside this ball, the geodesic distance on the embedded manifold and the Euclidean distance in the ambient space are equal up to an error which is higher order in $h_N$.  This fact follows directly for compact manifolds. Although it is not true for general manifolds, a weaker condition than compactness is sufficient, as first shown by \cite{heinThesis}.  We summarize this weaker condition by saying that a point $x$ on the manifold is {\it tangible} if the injectivity radius is nonzero, and the ratio between the  distance in the ambient space and the geodesic distance to nearby points is bounded away from zero. Theorem \ref{cor2} below shows that it is possible to consistently estimate the density at any tangible point in the interior of a manifold. This condition is explained in more detail in Appendix \ref{boundaryApp}.

The equivalence of the ambient and geodesic distances on small scales suggests that for kernels with sufficiently fast decay at infinity, the approaches of \cite{pelletier2005kernel,kim2013geometric,nips2009,heinThesis} are equivalent, with the exception of an additional bias term appearing in methods based on the Euclidean distance which depends on the extrinsic curvature of the embedding, as shown in \cite{heinThesis}.  This fact first came to light in \cite{diffusion}, although the focus was on estimating the Laplace-Beltrami operator on the unknown manifold, so the authors did not emphasize their density estimate result or analyze the variance of their estimate.  The fact was later pointed out in \cite{nips2009}, where the bias and variance of the kernel density estimate were computed.  The results of \cite{diffusion,heinThesis} apply to shape function kernels of the form $K\left( \frac{||x-X_i||}{h_N} \right)$, where $K:[0,\infty) \to [0,\infty)$ is assumed to have exponential decay.  We note that this includes all compactly supported kernels, such as the Epanechnikov kernel \cite{epanechnikov1969non} and other similar kernels that are often used in density estimation.  

Theorem \ref{cor2}, first shown in \cite{heinThesis}, shows that KDE is straightforward when $X_i$ are random variables sampled according to a density $f(x)$ on an embedded manifold with no boundary.  
\begin{thm}[KDE on Embedded Manifolds ]\label{cor2} Let $\tilde f$ be a density supported on an $m$-dimensional Riemannian manifold $\mathcal{M} \subset \mathbb{R}^n$ without boundary.  Let $\tilde f = f\, dV$ where $dV$ is the volume form on $\mathcal{M}$ inherited from the embedding and $f \in \mathcal{C}^4(\mathcal{M})$ is bounded above by a polynomial.  Let $K:[0,\infty)\to [0,\infty)$ have exponential decay and define $m_0 = \int_{\mathbb{R}^m} K\left(||z||\right) \, dz$.  If $x\in\mathcal{M}$ is a tangible point and $X_i$ are independent samples of $f$ then
\[ f_{h,N}(x) \equiv \frac{1}{N m_0 h^m}\sum_{i=1}^N K\left(\frac{||x-X_i ||}{h}\right) \]
is a consistent estimator of $f(x)$ with bias $\mathbb{E}\left[f_{h,N}(x) - f(x)\right] = \mathcal{O}(h^2)$ and variance $\textup{var}\left(f_{h,N}(x) - f(x)\right) = \mathcal{O}\left(\frac{h^{-m}}{N} f(x)\right)$.
\end{thm}

Theorem \ref{cor2} follows directly from Theorem \ref{thm2}, which is proved in Appendix \ref{boundaryApp}.  To find the density $f(x)$ written with respect to the volume form inherited from the embedding, Theorem \ref{cor2} implies
\begin{align}\label{standardKDE} f_h(x) \equiv \lim_{N\to\infty}\frac{1}{N m_0 h^m}\sum_{i=1}^N K\left(\frac{||x-X_i ||}{h}\right) = f(x) + \mathcal{O}(h^2). \end{align}
For example, one may use the Gaussian kernel
\[ K(z) = \pi^{-m/2}\exp\left(-||z||^2 \right), \]
which has $m_0 = 1$. Notice that \eqref{standardKDE} is simply a standard KDE formula in $\mathbb{R}^n$, except that $h^m$ appearing in the denominator would normally be $h^n$.  Intuitively, this is because the data lies on an $m$-dimensional subspace of the $n$-dimensional ambient space.

\section{Boundary Correction}\label{boundary}

The topic of kernel density estimation near a boundary has been thoroughly explored in case the distribution is supported on a subinterval $[b, \infty]$ of the real line, and with the assumption that $b$ is known. Using a naive kernel such as \eqref{standardKDE} results in an estimate that is not even consistent near $b$. An early method that achieved consistency on the boundary was the cut-and-normalize method \cite{gasser1979kernel}, although the bias was only order $h$ on the boundary, despite being order $h^2$ in the interior. Various alternatives were proposed to obtain bias uniform over the domain and boundary. 
These methods include reflecting the data \cite{schuster1985incorporating,karunamuni2008some}, generalized jackknifing \cite{boundary93,boundary96,boundary99,kyung1998nonparametric}, translation-based methods \cite{hall2002new}, and the use of specially-designed asymmetric kernels from the beta and gamma distributions \cite{boundary00,boundary05,boundary14}. The cut-and-normalize method was extended to order $h^2$ on the boundary in \cite{kyung1998nonparametric}.

The goal of this section is to generalize the cut-and-normalize approach to an order $h^2$ method, including boundary and interior, on embedded manifolds {\it where the location of the boundary is not known}.
The standard KDE formula \eqref{standardKDE} fails for manifolds with boundary because the domain of integration is no longer symmetric near the boundary.  For a point $x$ on the boundary $\partial \mathcal{M}$, the integral over the local region $N_{h}(x)$ approximates the integral over the half space $T_x\mathcal{M} \cong \mathbb{R}^{m-1} \oplus \mathbb{R}^+$.   The zeroth moment of local kernel $m_0(x)$ is defined to be the integral over $\mathbb{R}^m$, so dividing by this normalization constant will lead to a biased estimator even in the limit $h\to 0$.  While technically the estimator is still asymptotically unbiased for all interior points, for fixed $h$ the additional bias from using the incorrect normalization constant can be quite large for points within geodesic distance $h$ of the boundary. 

To fix the bias, we will estimate the distance $b_x$ and direction $\eta_x$ to $\partial \mathcal{M}$ for every point $x$ in $\mathcal{M}$. Our motivation is that if they are known, Theorem \ref{thm2} below gives a consistent estimate of the density $f(x)$ both in the interior and the boundary. Next, we compute three more variants of the KDE computation \eqref{standardKDE} to estimate $b_x$ and $\eta_x$, and to extend the  second-order estimate of $f(x)$ everywhere. Figure \ref{workflow} shows the proposed workflow.

 \begin{figure}[h] 
\begin{center}
\includegraphics[width=.7\linewidth]{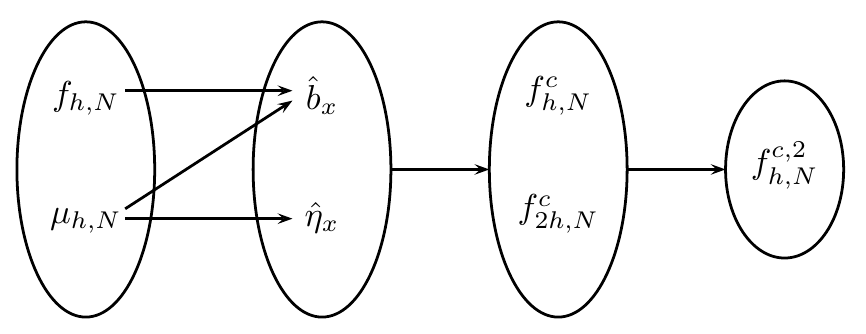}
\end{center}
\caption{Workflow schematic. At each point $x$, the standard KDE \eqref{standardKDE} $f_{h,N}$ is combined with the boundary direction estimator BDE \eqref{bde} $\mu_{h,N}$ to estimate the distance $b_x$ to $\partial \mathcal{M}$. The boundary direction $\eta_x$ is estimated by the unit vector $\hat{\eta}_x$ in the direction of $\mu_{h,N}$. Cut-and-normalize estimators $f^c_{h,N}$ and $f^c_{2h,N}$ are calculated and combined to get the second-order estimate $f^{c,2}_{h,N}$. \label{workflow}}
\end{figure}

First,  in section \ref{distdirection} we compute the boundary direction estimator (BDE) denoted
\begin{equation} \label{bde} \mu_{h,N}(x) \equiv \frac{1}{N h^{m+1}}\sum_{i=1}^N K\left(\frac{||x-X_i ||}{h}\right)(X_i-x). 
\end{equation}
The BDE is sensitive to the presence of the boundary, and we will combine the KDE \eqref{standardKDE} with the BDE \eqref{bde} to derive estimates of $b_x$ and $\eta_x$.

Second, with $b_x$ and $\eta_x$ known, in section \ref{cutnormalize} we approximate $\partial \mathcal{M}$ as a hyperplane in the tangent space to more accurately normalize a cut-and-normalize kernel denoted $f^c_{h,N}$. Third, section \ref{extrap} repeats the cut-and-normalize kernel with bandwidth $2h$, so that Richardson extrapolation can be used to decrease the order of the error of $f(x)$ to $O(h^2)$ at points $x$ up to and including the boundary.

\subsection{Distance and Direction to the Boundary}\label{distdirection}

Correcting the bias of the standard KDE \eqref{standardKDE} near the boundary requires computing the true zeroth moment of the kernel. Equation \eqref{KDEb} is an adjusted version of \eqref{standardKDE} with the correct normalization.

\begin{thm}[KDE near the Boundary]\label{thm2} Assume the hypotheses of Theorem \ref{cor2} except with $\mathcal{M}$ a manifold with boundary $\partial \mathcal{M}$ of bounded curvature.  Let $x \in \mathcal{M}$ be tangible, let $\dM_x$ be the geodesic distance to the boundary, and let $\eta_x \in T_x\mathcal{M}$ be a unit vector in the direction of the boundary.  Then for $h$ sufficiently small, $\eta_x$ is well defined and
\begin{equation} f^{\partial}_{h,N}(x) \equiv \frac{1}{N m^{\partial}_0(x) h^m}\sum_{i=1}^N K\left(\frac{||x-X_i ||}{h}\right) \label{KDEb}
\end{equation}
is a consistent estimator of $f(x)$. Here
\[ m^{\partial}_0(x) = \int_{\mathbb{R}^{m-1}}\int_{-\infty}^{\dM_x/h} K(||z_{\perp}|| + |z_{\parallel}|) \, dz_{\parallel} dz_{\perp} \]
where $z_{\perp} \perp \eta_x$ and $z_{\parallel}$ is a scalar.  Moreover, $f^{\partial}_{h,N}(x)$ has bias
\begin{equation} \label{kdebias}\mathbb{E}\left[f^{\partial}_{h,N}(x) - f(x)\right] = h m^{\partial}_1(x) \eta_x \cdot \nabla f(x) + \mathcal{O}(h^2) 
\end{equation}
and variance
\[ \textup{var}\left(f^{\partial}_{h,N}(x) - f(x)\right) = \frac{h^{-m}}{N} \frac{m^{2,\partial}_0(x)}{m^{\partial}_0(x)}f(x) + \mathcal{O}(1/N), \] where
\[ m_1^{\partial}(x) = -\int_{\mathbb{R}^{m-1}}\int_{-\infty}^{\dM_x/h} K(||z_{\perp}|| + |z_{\parallel}|) z_{\parallel} \, dz_{\parallel} dz_{\perp} \]
and
\[ m_0^{2,\partial}(x) = \int_{\mathbb{R}^{m-1}}\int_{-\infty}^{\dM_x/h} K\left(||z_{\perp}|| + |z_{\parallel}|\right)^2 \, dz_{\parallel} dz_{\perp}. \]
\end{thm}

The proof of Theorem \ref{thm2} is in Appendix \ref{boundaryApp}.  Intuitively, Theorem \ref{thm2} says that finding a consistent estimator of $f(x)$ for points near the boundary requires correcting the zeroth moment $m_0$.  For interior points, the zeroth moment is the integral of the kernel over the entire tangent space, but for boundary points, the integral only extends to the boundary.  Since we choose an orientation with $\eta_x$ pointing towards the boundary (for boundary points $\eta_x$ is the unit normal vector), the integral over $z_{\parallel}$ extends infinitely in the negative direction (into the interior of the manifold) but only up to $\dM_x/h$ in the positive direction (toward the boundary).  One should think of $h z_{\parallel}\eta_x$ being a tangent vector that extends up to $\dM_x$, which explains why $z_{\parallel}$ extends to $\dM_x/h$.  Finally, notice that for $dM(x) \gg h$, the zeroth moment $m^{\partial}_0(x)$ reduces to the zeroth moment $m_0(x)$ for manifolds without boundary up to an error of higher order in $h$ due to the decay of the kernel.  This shows that the estimator (\ref{KDEb}) of Theorem \ref{cor2} is consistent for all interior points.  However, for a fixed $h$, the bias will be significantly larger using the estimator of (\ref{standardKDE}) than for the estimator of Theorem \ref{thm2} for points with $\dM_x \leq h$.

The applicability of \eqref{KDEb} depends on an efficient calculation of $m^{\partial}_0(x)$. For general local kernels, the formula for $m^{\partial}_0(x)$ can be very difficult to evaluate near the boundary.  A possible solution is to apply an asymptotic expansion in $\dM_x/h$, for example,
\begin{align} m^{\partial}_0(x) &= \int_{\mathbb{R}^{m-1}}\int_{-\infty}^{0} K(||z_{\perp}|| + |z_{\parallel}|) \, dz_{\parallel} dz_{\perp} + \frac{\dM_x}{h}  \int_{\mathbb{R}^{m-1}} K(||z_{\perp}||) \, dz_{\perp} + \mathcal{O}\left(\left(\frac{\dM_x}{h}\right)^2\right).  \nonumber
\end{align}
However, working with these asymptotic expansions is very complicated.  Moreover, the asymptotic expansion suggests a fundamental connection between $m^{\partial}_0(x)$ and the standard zeroth moment $m_0(x)$ for an $(m-1)$-dimensional manifold.  Exploiting this connection requires a kernel which can convert the sum $h ||z_{\perp}|| + h |z_{\parallel}|$ into a product.  Of course, the only kernel which can make this separation exactly is the exponential kernel,
\begin{equation} \label{expker} K(z) = \pi^{-m/2}\exp\left(-||z||^2 \right) \end{equation}
where we have,
\[ K(z_{\perp} + z_{\parallel}\eta_x) = \pi^{-(m-1)/2}\exp\left(-||z_{\perp}||^2\right) \pi^{-1/2}\exp\left(-z_{\parallel}^2\right). \]
This property dramatically simplifies KDE for manifolds with boundary, as shown by the following explicit computation,
\begin{align}\label{gaussianm} m^{\partial}_0(x) &= \pi^{-(m-1)/2} \int_{\mathbb{R}^{m-1}} \exp\left(-||z_{\perp}||^2\right) \, dz_{\perp} \pi^{-1/2} \int_{-\infty}^{\dM_x/h} \exp\left(-z_{\parallel}^2\right) \, dz_{\parallel}  \nonumber \\
& = \pi^{-1/2} \int_{-\infty}^{0} \exp\left(-z_{\parallel}^2\right) \, dz_{\parallel} + \pi^{-1/2} \int_{0}^{\dM_x/h} \exp\left(-z_{\parallel}^2\right) \, dz_{\parallel}  \nonumber \\
&= \frac{1}{2}(1 + \textup{erf}(\dM_x/h))   
\end{align}
Due to this simplification, we advocate the exponential kernel for KDE on manifolds with boundary.

Making use of Theorem \ref{thm2} with $m^{\partial}_0(x)$ from \eqref{gaussianm} reduces the problem to estimating the distance $b_x$ to the boundary. 
The next theorem calculates the expectation of the BDE estimator \eqref{bde}
\[ \mu_{h,N}(x) \equiv \frac{1}{N h^{m+1}}\sum_{i=1}^N K\left(\frac{||x-X_i ||}{h}\right)(X_i-x). \]
Together with Theorem \ref{thm2}, the BDE estimator will be used to estimate the distance $b_x$.

\begin{thm}[Boundary Direction Estimation]\label{thm3} Under the same hypotheses as Theorem \ref{thm2}, $\mu_{h,N}(x)$ has expectation
\[ \mathbb{E}[\mu_{h,N}(x)] = -\eta_x f(x) m_1^{\partial}(x) + \mathcal{O}(h \nabla f(x), h f(x)) \]
where $\eta_x \in T_x\mathcal{M}$ is a unit vector pointing towards the closest boundary point (for $x\in \partial\mathcal{M}$, $\eta_x$ is the outward pointing normal). 
\end{thm}

The proof of Theorem \ref{thm3} is in Appendix \ref{boundaryApp}. Notice that the minus sign in the definition of $m_1^{\partial}(x)$ implies that for most kernels, $m_1^{\partial}(x) > 0$ (since the integral is heavily weighted toward the negative $z_{\parallel}$ direction).  This choice of minus sign gives the correct impression that $\mu_{h,N}(x)$ points into the interior (the opposite direction of $\eta_x$).  

Intuitively, Theorem \ref{thm3} follows from the fact that the integrand $K(||z||)z$ of the BDE estimator is odd, and the domain of integration is symmetric in every direction $z \perp \eta_x$.  The only non-symmetric direction is parallel to $\eta_x$ due to the boundary.  Thus, the integral is zero in every direction except $-\eta_x$, where the minus sign follows from the fact that there are more points in the interior direction than in the boundary direction (since the boundary cuts off the data).  Of course, it is possible for a large density gradient to force $\mu$ to point in a different direction, which explains the bias term of order $h\nabla f(x)$, but this is a higher order error.  

For the Gaussian kernel \eqref{expker}, we again have an exact expression for the integral
\begin{align}\label{gaussianmu}\mathbb{E}[\mu_{h,N}(x)] &= \eta_x f(x) \pi^{-1/2} \int_{-\infty}^{\dM_x/h} \exp\left(-z_{\parallel}^2\right) z_{\parallel} \, dz_{\parallel} = -\eta_x \frac{f(x)}{2\sqrt{\pi}} \exp\left(-\frac{\dM_x^2}{h^2}\right)
\end{align}
and we will now use this expression combined with \eqref{gaussianm} to find $\dM_x$.  Since $f(x)$ is unknown, and appears in both $f_{h,N}(x)$ and $||\mu_{h,N}(x)||$, the natural quantity to consider is
\begin{align}
\frac{f_{h,N}(x)}{\sqrt{\pi}||\mu_{h,N}(x)||} = \left(1+\textup{erf}\left( \dM_x/h \right) \right)e^{\dM_x^2/h^2}. \nonumber
\end{align} 
In order to find $\dM_x$ we will solve the above expression numerically by setting $c =  \frac{f_{h,N}(x)}{\sqrt{\pi}||\mu_{h,N}(x)||}$ and defining
\[ F(\dM_x) = \left(1+\textup{erf}\left( \dM_x/h \right) \right)e^{\dM_x^2/h^2} - c, \]
where we note that
\[ F'(\dM_x) =\frac{2}{\sqrt{\pi}h} + 2\left(1+\textup{erf}\left( \dM_x/h \right) \right)e^{\dM_x^2/h^2} \frac{\dM_x}{h^2},   \]
Newton's method can be used to solve $F(\dM_x) = 0$ for $\dM_x$.  In fact, using the fact that $1 \leq 1+\textup{erf}(\dM_x/h) < 2$, a very simple lower bound for $\dM_x$ is
\[ \dM_x \geq h \sqrt{\max\{0,-\log(c/2)\}} \]
and this can be a useful initial guess for Newton's method.

Finally, using the estimated value for $\dM_x$ we can evaluate $m_0(x) = \frac{1}{2}\left(1+ \textup{erf}\left(\dM_x/h\right)\right)$ and use the KDE formula in Theorem \ref{thm2} with this $m_0(x)$, which yields a consistent estimator of $f_{h,N}(x)$ on manifolds with boundary.  

\begin{examp}[KDE on a Disc]\label{example1}  \rm In this example we verify the above expansions for data sampled on the disk $D^2 = \{(r\cos\theta,r\sin\theta)\in\mathbb{R}^2 : r\leq 1\}$ according to the density $f(r,\theta) = \frac{2}{3\pi}(2-r^2)$.  In order to generate samples $x_i = (r_i,\theta_i)$ from the density $f$, we use the rejection sampling method.  We first generate points on the disc sampled according to the uniform density $f_0(r,\theta) = \textup{vol}(D^2)^{-1} = \pi^{-1}$ by generating 12500 uniformly random points in $[-1,1]^2$ and then eliminating points with distance to the origin greater than 1.  Next we set $M = \max_{r,\theta} \{ f(r,\theta)/f_0(r,\theta) \} = 4/3$ and for each uniformly sampled point $\tilde x_i$ on $D^2$, we draw a uniformly random variable $\xi_i \in [0,1]$ and we reject the $i$-th point if $\xi_i \geq \frac{f(\tilde x_i)}{M f_0(\tilde x_i)} = 1-r^2/2$ and otherwise we accept the point as a sample $x_i$ of $f$.  In this experiment there were $N=7316$ points remaining after restricting to the unit disc and rejection sampling.

 \begin{figure}[h]
\begin{center}
\includegraphics[width=.49\linewidth]{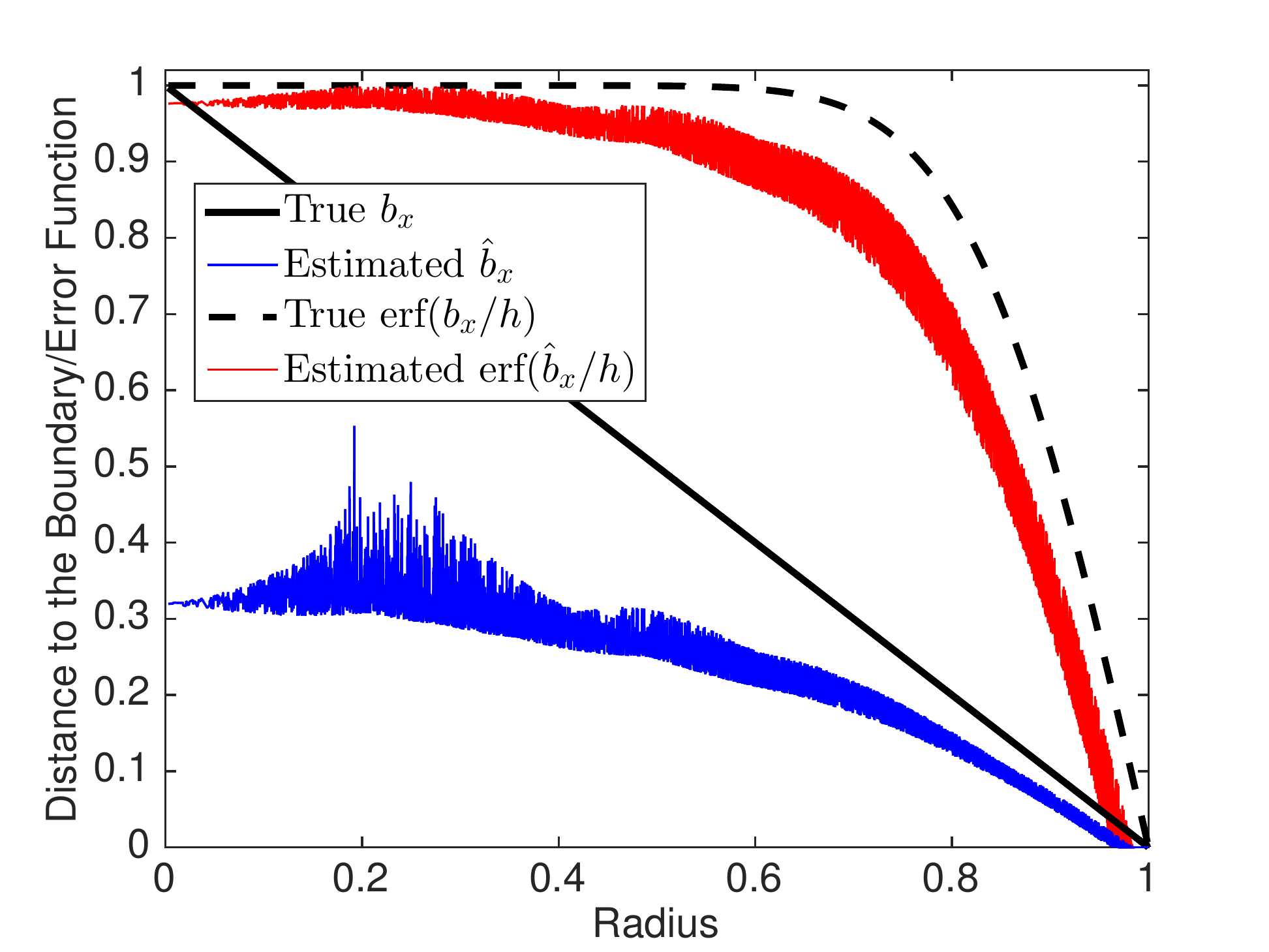}
\includegraphics[width=.49\linewidth]{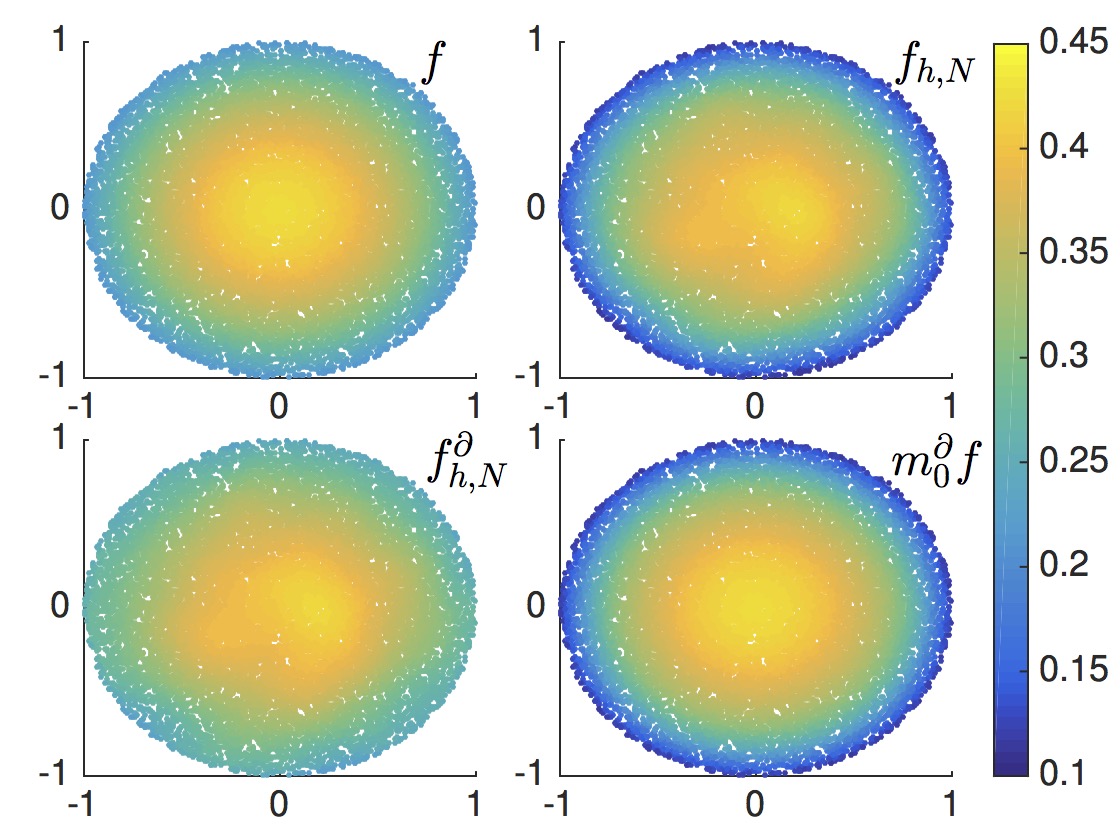}
\end{center}
\caption{\label{ex1} Verifying the estimation of the distance to the boundary with $h=0.2$ on the disk data set.  Left: The true distance to the boundary (black, solid curve) is compared to the recovered distance (blue) is shown for each point as a function of the radius, we also show the value of $\textup{erf}(\dM_x/h)$ for both the true distance (black, dashed curve) and the recovered distance (red). Right: The true density $f$ is compared to the standard KDE $f_{h,N}$ and the boundary correction $f_{h,N}^{\partial}$ as well as the theoretical standard KDE result $m_0^{\partial} f$.}
\end{figure}

Using the data $x_i$, we first evaluate the standard KDE formula (without boundary correction)
\[ f_{h,N}(x_i) = \frac{1}{N h^2} \sum_{j=1}^N K\left(\frac{||x_i-x_j ||}{h}\right) \] 
and
\[ \mu_{h,N}(x_i) =  \frac{1}{N h^3} \sum_{j=1}^N K\left(\frac{||x_i-x_j ||}{h}\right)(x_j-x_i) \] 
on each data point.  In this example we use the standard Gaussian kernel described above.  In order to correct the KDE on the boundary, we first estimate the distance to the boundary using the strategy outlined above, and the results of this estimate are shown in Figure \ref{ex1} (top, left).  We then compute $m^{\partial}_0(x)$ which allows us to compute the boundary correction $f^{\partial}_{h,N}(x)$ and in Figure \ref{ex1} (top, right) we compare this to the standard KDE $f_{h,N}(x)$ as well as the true density $f$ and the theoretically derived large data limit $m^{\partial}_0(x)f(x)$ of the standard KDE.  These quantities are also compared qualitatively on the disc in Figure \ref{ex1}, which clearly shows the underestimation of the standard KDE on the boundary, which also agrees with the theoretically derived standard KDE result which is $m^{\partial}_0(x)f(x)$.  In contrast, the boundary correction $f_{h,N}^{\partial}$ slightly overestimates the true density on the boundary, due to $h = 0.2$ being quite large in this example.

\end{examp}

\subsection{The Cut and Normalize Method}\label{cutnormalize}

The weakness of the previous approach is that the estimate of $\dM_x$ may not be very accurate, especially for points far from the boundary.  Of course, since the function $\textup{erf}(\dM_x/h)$ saturates for $\dM_x$ sufficiently large, this somewhat ameliorates the problem of underestimating $\dM_x$.  However, it would be preferable in terms of bias to have an exact value for $\dM_x$.  In fact, the kernel weighted average $\mu_{h,N}(x)$ makes this possible.  Notice that the unit vector in the direction of $-\mu_{h,N}(x)$ is an estimate of $\eta_x$, namely
\[ \hat{\eta}_x \equiv \frac{-\mu_{h,N}(x)}{||\mu_{h,N}(x)||} = \eta_x + \mathcal{O}(h). \]
Since $\mu_{h,N}(x)$ tells us the direction of the boundary, we can protect against underestimation of $\dM_x$ by actually cutting of the kernel at the estimated distance to the boundary.  

Given an estimate $\hat{\dM}_x$ of $\dM_x$, the cut-and-normalize method only includes samples $X_i$ such that $X_i \cdot \hat\eta_x \leq \hat{\dM}_x$, which gives us the following estimator,
\begin{equation} \label{KDEcut} f^{c}_{h,N}(x) \equiv \frac{1}{N (1+\textup{erf}(\hat{\dM}_x/h)) h^m}\sum_{X_i \cdot \hat\eta_x \leq \hat{\dM}_x} K\left(\frac{||x-X_i ||}{h}\right) 
\end{equation}
which is a consistent estimator for any $0<\hat{\dM}_x < \dM_x$.  Of course, this cut-and-normalize method has several potential downsides.  The first is that by not including the maximum possible number of points, we have increased the variance of our estimator.  The second is that for points in the interior, the cut-and-normalize method may eliminate the symmetry of the region of integration, leading to increased bias for interior points.  However, as long as the estimate $\dM_x$ is larger than $h$ for points that are far from the boundary, the effect of cutting the domain outside of $h$ will be negligible (see Lemma \ref{lem1} for details).  In our empirical investigations, we have found that the error introduced by the cut-and-normalize method is very small compared to the error of using an incorrect estimate of $\dM_x$ direction in $m^{\partial}_0(x)$.  In Figure \ref{HOdisk} we apply the cut-and-normalize method to Example \ref{example1} and show that for interior points, the method produces results that are comparable to the standard KDE.  This should be compared with Figure \ref{ex1} which simply renormalizes using the estimated distance to the boundary without cutting, and does not match the standard KDE for interior points.

 \begin{figure}[h]
\begin{center}
\includegraphics[width=.51\linewidth]{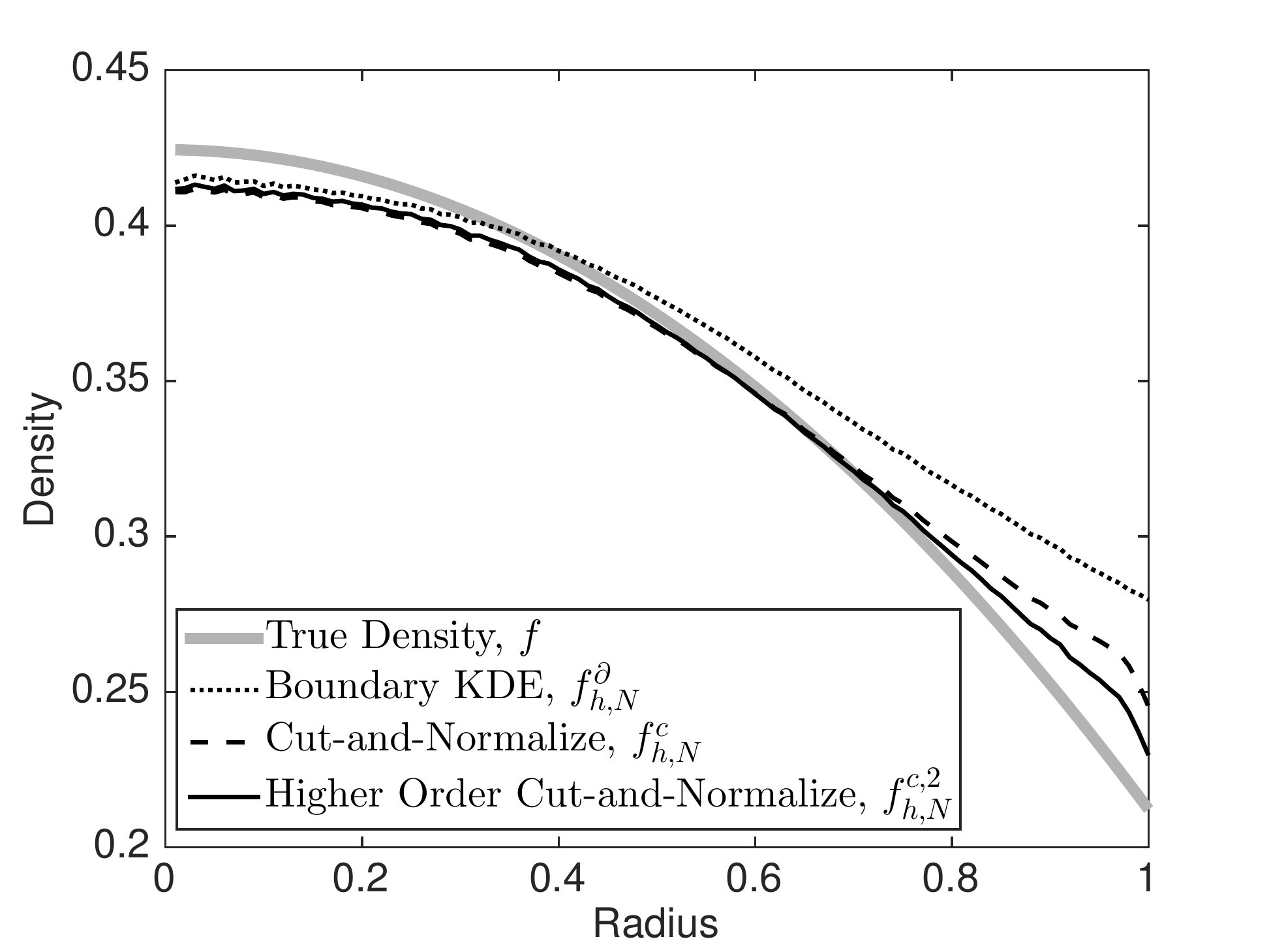}
\includegraphics[width=.48\linewidth]{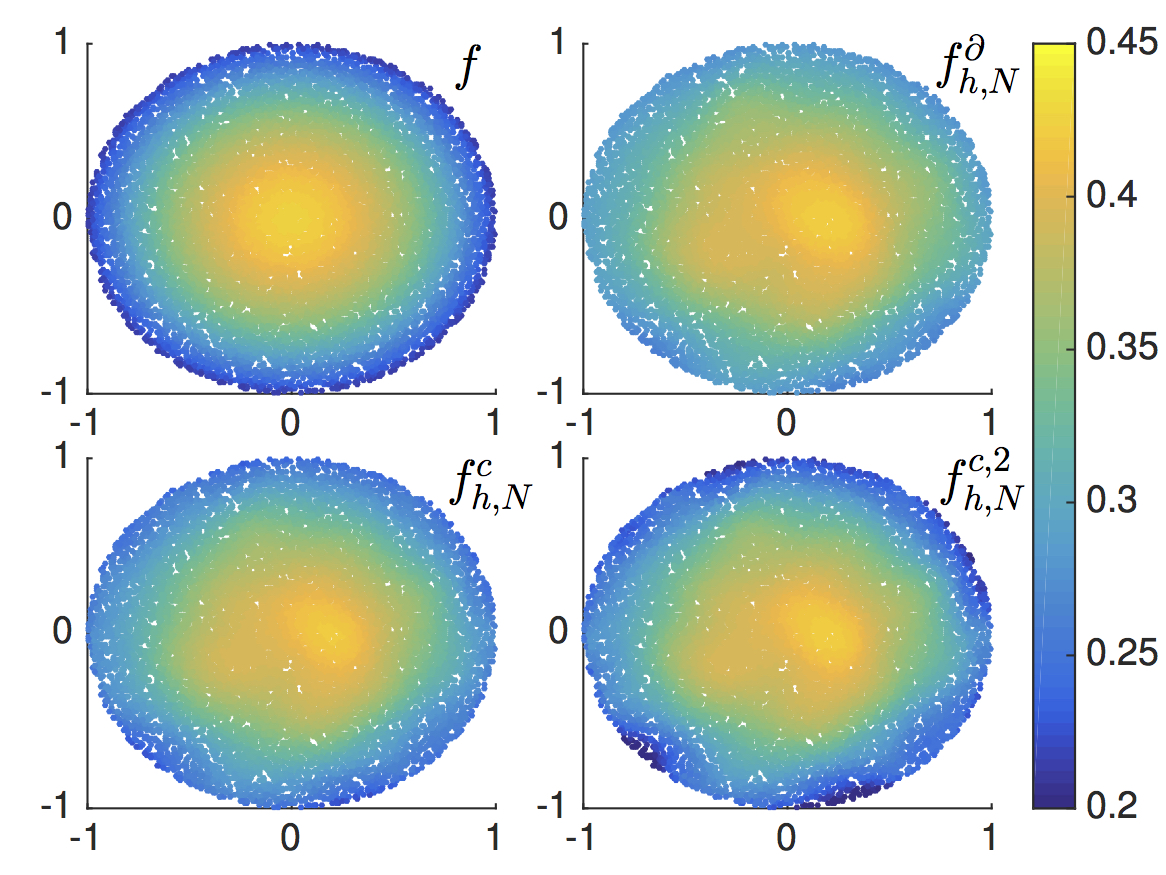}
\end{center}
\caption{\label{HOdisk} Comparison of the boundary correction, cut-and-normalize method, and higher order cut-and-normalize method on the disk of Example \ref{example1}.  Left: Average value of each density estimation method as a function of radius after repeating the experiment of Example \ref{example1} 10 times with independent data sets which shows the bias (variance error is averaged out).  Notice that the higher order cut-and-normalize method $f_{h,N}^{c,2}$ has similar bias on the boundary and the interior, which is order $h^2=0.04$ in each case.  Right: A single realization of the experiment in Example \ref{example1} showing the true density and all three density estimates (note that the color scale is different than in Figure \ref{ex1} to better show the differences in these estimates).}
\end{figure}

\subsection{Higher-Order Boundary Correction}\label{extrap}

The above method obtains an asymptotically unbiased estimate of the sampling density at all points of the manifold, including the boundary.  However, the bias in the interior of the manifold is $\mathcal{O}(h^2)$, which is significantly smaller than for points very near the boundary, where the bias is $\mathcal{O}(h)$.  In order to obtain a uniform rate of convergence at all points, we need to eliminate the order-$h$ term $h m_1^{\partial}(x)\eta_x \cdot \nabla f(x)$ appearing in the bias of Theorem \ref{thm2}.  

To construct a higher-order kernel we will use Richardson extrapolation, which is a general method of combining estimates from multiple values of $h$ to form a higher order method. Its use is common in the kernel density estimation literature \cite{schucany1977improvement,boundary96,kyung1998nonparametric}.  Our goal is to cancel the bias term \eqref{kdebias}
\[ h m^{\partial}_1(x) \eta_x \cdot \nabla f(x) = h (1+\textup{erf}(\dM_x/h))e^{-\dM_x^2/h^2} \eta_x \cdot \nabla f(x) \] 
using a linear combination of two KDE formulas with different values of $h$.  Consider the bias for bandwidths $h$ and $2h$:
\begin{equation}\mathbb{E}[f^c_{h,N}(x)] = f(x) +  h (1+\textup{erf}(\dM_x/h))e^{-\dM_x^2/h^2} \eta_x \cdot \nabla f(x) + \mathcal{O}(h^2) 
\end{equation}
\begin{equation} \mathbb{E}[f^c_{2h,N}(x)] = f(x) +  2 h (1+\textup{erf}(\dM_x/(2h)))e^{-\dM_x^2/(4h^2)} \eta_x \cdot \nabla f(x) + \mathcal{O}(h^2).
 \end{equation}
Set ${\displaystyle C=\frac{(1+\textup{erf}(\dM_x/(2h)))e^{-\dM_x^2/(4h^2)}}{(1+\textup{erf}(\dM_x/h))e^{-\dM_x^2/(h^2)}}}$ and define the second-order cut-and-normalize density estimator as
\begin{equation} \label{rich} f_{h,N}^{c,2}(x) \equiv \frac{2Cf^c_{h,N}(x)-f^c_{2h,N}(x)}{2C-1}.
\end{equation}
The order-$h$ term of the bias cancels, so that
\[ \mathbb{E}[f_{h,N}^{c,2}(x)] = f(x) + \mathcal{O}(h^2), \]
which is the same asymptotic bias as the standard KDE in Corollary \ref{cor2} for embedded manifolds without boundary.  It is also interesting to note that as $\dM_x$ becomes larger than $h$, the higher-order formula reduces to $f^c_{2h,N}$.  This shows that this kernel is only ``higher-order" on the boundary, and in fact is the same order as the standard KDE on the interior, so in fact $f_{h,N}^{c,2}(x)$ has a bias which is order-$h^2$ on the boundary and the interior.  

The higher order cut-and-normalize method KDE is implemented in the examples below and show bias that is significantly reduced compared to the naive cut-and-normalize method.  Figure \ref{workflow} summarizes the complete algorithm. At a given point $x$ and for a given bandwidth $h$ and number of data points $N$, we compute the KDE \eqref{standardKDE} and BDE \eqref{bde} to estimate $f_{h,N}$ and $\mu_{h,N}$, respectively. These estimators combine to produce the estimates $\hat{b}_x$ and $\hat{\eta}_x$ for the distance and direction to the boundary, respectively. Finally, we compute the cut-and-normalize estimator \eqref{KDEcut} for $h$ and $2h$ and extrapolate the estimates using \eqref{rich} to arrive at the final second-order estimate $f_{h,N}^{c,2}(x)$ throughout the manifold with boundary.

We first consider an example on a noncompact manifold with boundary, namely a Gaussian distribution restricted to a half-plane.  The manifold in this case is the entire half-plane, which is a simple linear manifold with infinite injectivity radius (see note in Appendix \ref{boundaryApp}) and $R(x,y) = 1$ for all pairs of points.  This means that the half-plane is a uniformly tangible manifold and so we can estimate the density effectively at each point of a sample set.  

\begin{examp}[Gaussian in the Half-Plane] \rm

We generated 20000 points from a standard 2-dimensional Gaussian and then rejected all the points with first coordinate less than zero.  Setting $h=\sqrt{0.06}$, the standard KDE formula $f_{h,N}$ and the BDE $\mu_{h,N}$ were computed. Then the cut-and-normalize estimator $f_{h,N}^c$ and the second-order cut-and-normalize estimator $f_{h,N}^{c,2}$ were calculated as in the flowchart of Figure \ref{workflow}.  These estimates are compared in Figure \ref{HOhalfGaussian}.  Notice that the standard KDE moves the mode of the distribution into the right half-plane, whereas both cut-and-normalize methods yield a mode very close to zero. Of course, the input to the algorithm are the data points only; no information about the manifold is assumed known.

 \begin{figure}[h]
\begin{center}
\includegraphics[width=.51\linewidth]{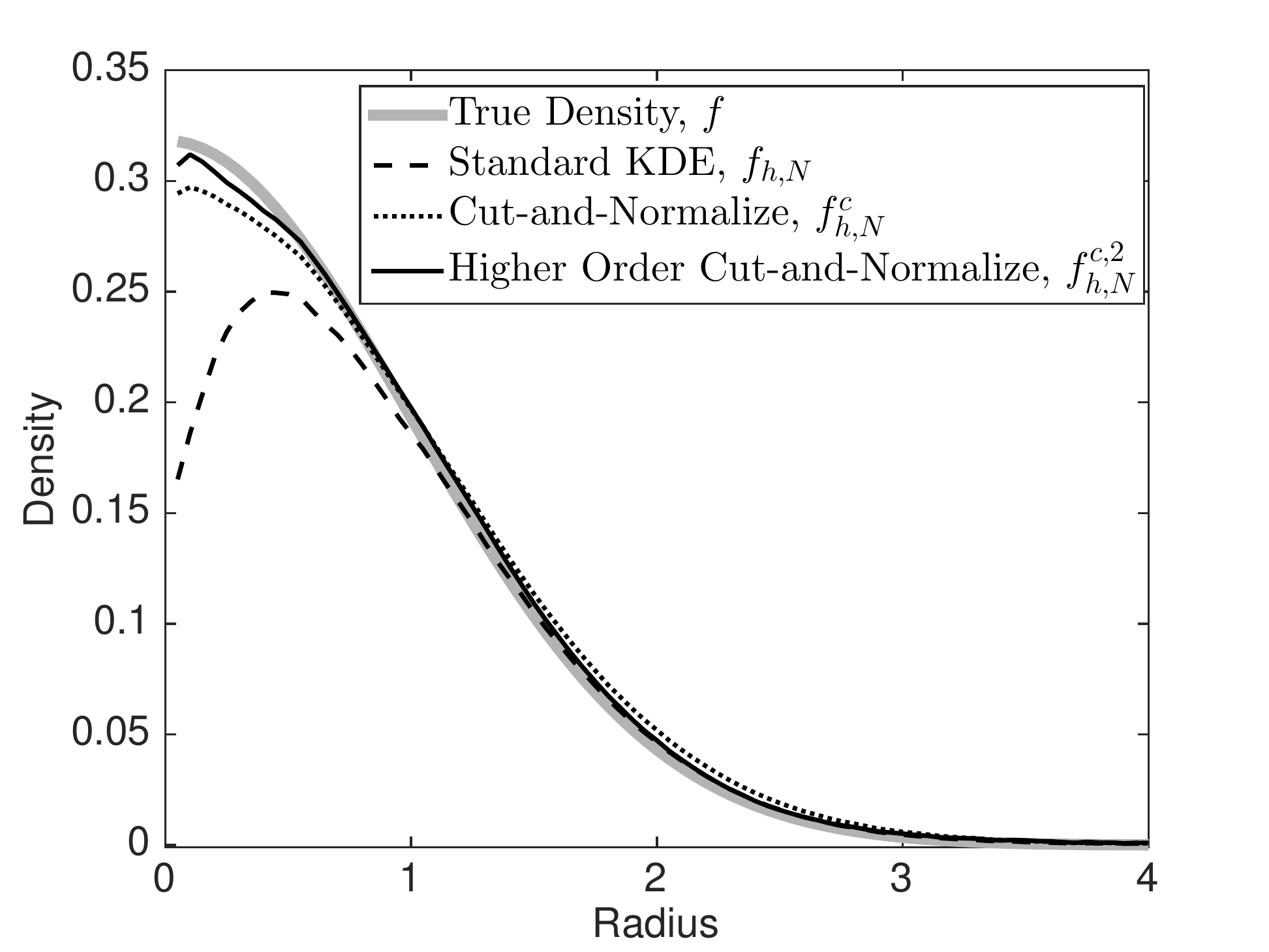}
\includegraphics[width=.48\linewidth]{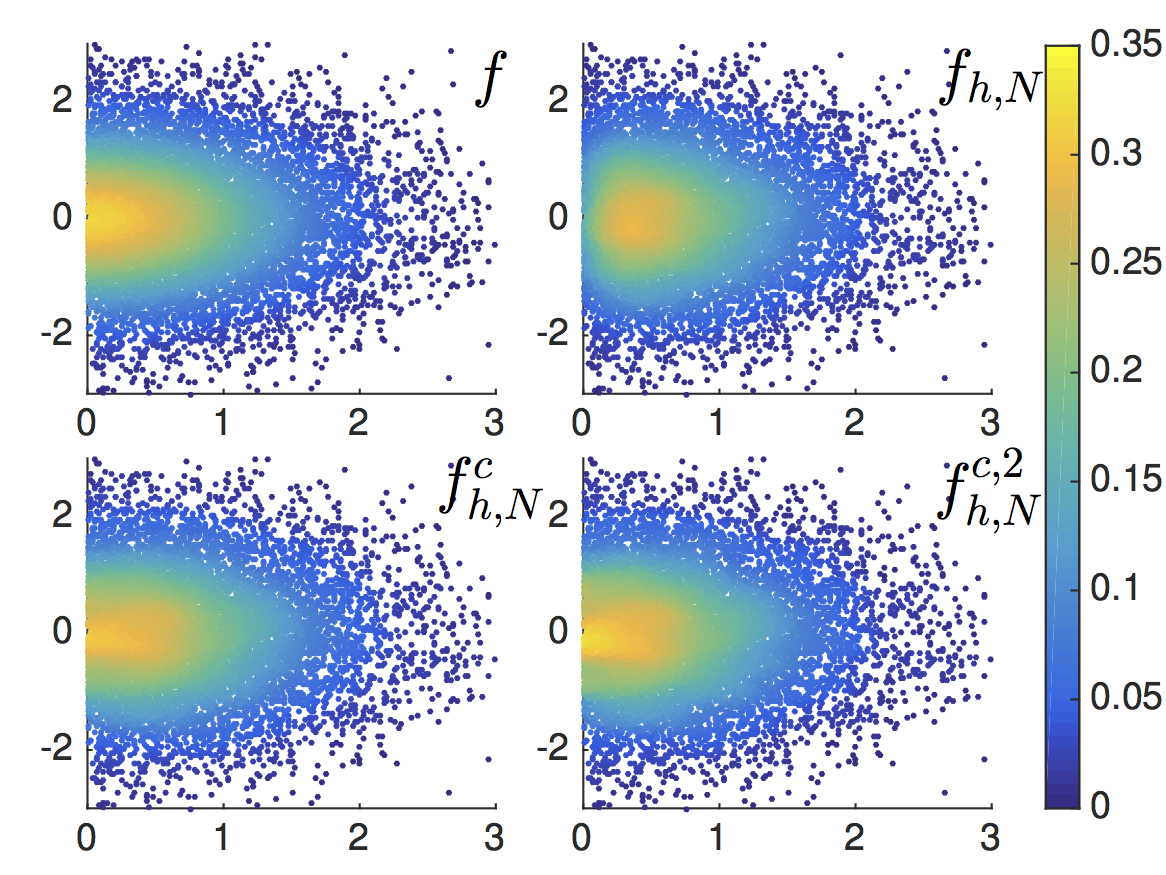}
\end{center}
\caption{\label{HOhalfGaussian} Comparison of the standard KDE, cut-and-normalize method, and higher order cut-and-normalize method on the Gaussian restricted to the half-plane.  Left: Average value of each density estimation method as a function of radius after repeating the experiment 10 times with independent data sets which shows the bias (variance error is averaged out).  Right: A single realization of the experiment showing the true density and all three density estimates.}
\end{figure}

\end{examp}

The next example demonstrates the benefits of the higher-order boundary correction on a portion of a hemisphere with an oscillating boundary (see Figure \ref{HOhemi}).  This manifold is particularly difficult for density estimation due to the large curvature of the boundary.  For a point in the middle of one of the arms, there are two boundaries which are equidistant apart.  Of course, in the limit of very small $h$, these points will not be able to see either boundary, but for large $h$ this can lead to significant bias.  

\begin{examp}[Hemisphere with Oscillating Boundary]\rm

To generate this data set, we began by sampling 50000 points uniformly from $[-1,1]^2$ in the plane, and keep only the points with
\[ r \leq \sin\left(6\left(\theta-\frac{\pi}{12}\right)\right)/8 + \frac{3}{4}, \]
which gives a subset of the disk of radius $7/8$ with an oscillating boundary.  A $z$-coordinate on the unit sphere is assigned to each point by setting $z = \sqrt{1-x^2 + y^2}$.  The volume form is given by  $dV = \textup{det}(D\mathcal{H}^\top D\mathcal{H})^{1/2}$ where $\mathcal{H}:(x,y) \mapsto (x,y,\sqrt{1-x^2-y^2})$ which is $dV = (1-x^2-y^2)^{-1/2}$.  Thus, by mapping uniformly sampled points from the disk onto the hemisphere, the sampling measure of the data at this point is proportional to $dV^{-1} = \sqrt{1-x^2-y^2}$. 

 To normalize the distribution, this function $dV^{-1}$ is integrated against the volume form $dV$, and in polar coordinates $r = \sqrt{x^2+y^2}$ the integral is
\[ \int_0^{2\pi}\int_0^{\sin(6\theta-\pi/2)/8 + 3/4} r dr  d\theta = \frac{73\pi}{128}. \]
The initial density is $q(r) = \frac{128}{73\pi}\sqrt{1-r^2}$. This density is largest in the interior, and the density gradient helps to insure that $\mu_{h,N}$ points in the correct direction (into the interior of the manifold).

 \begin{figure}[h]
\begin{center}
\includegraphics[width=.48\linewidth]{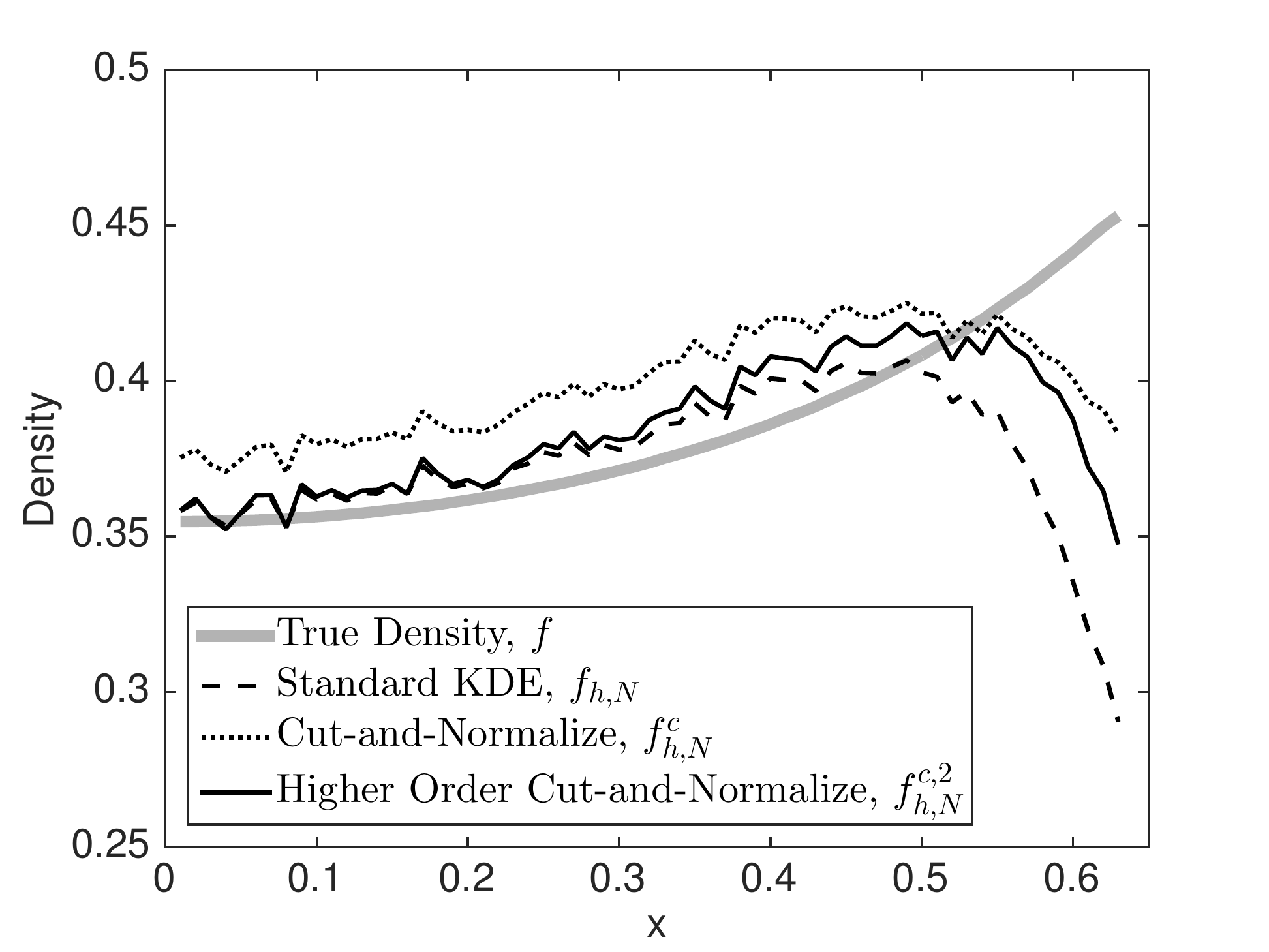}
\includegraphics[width=.48\linewidth]{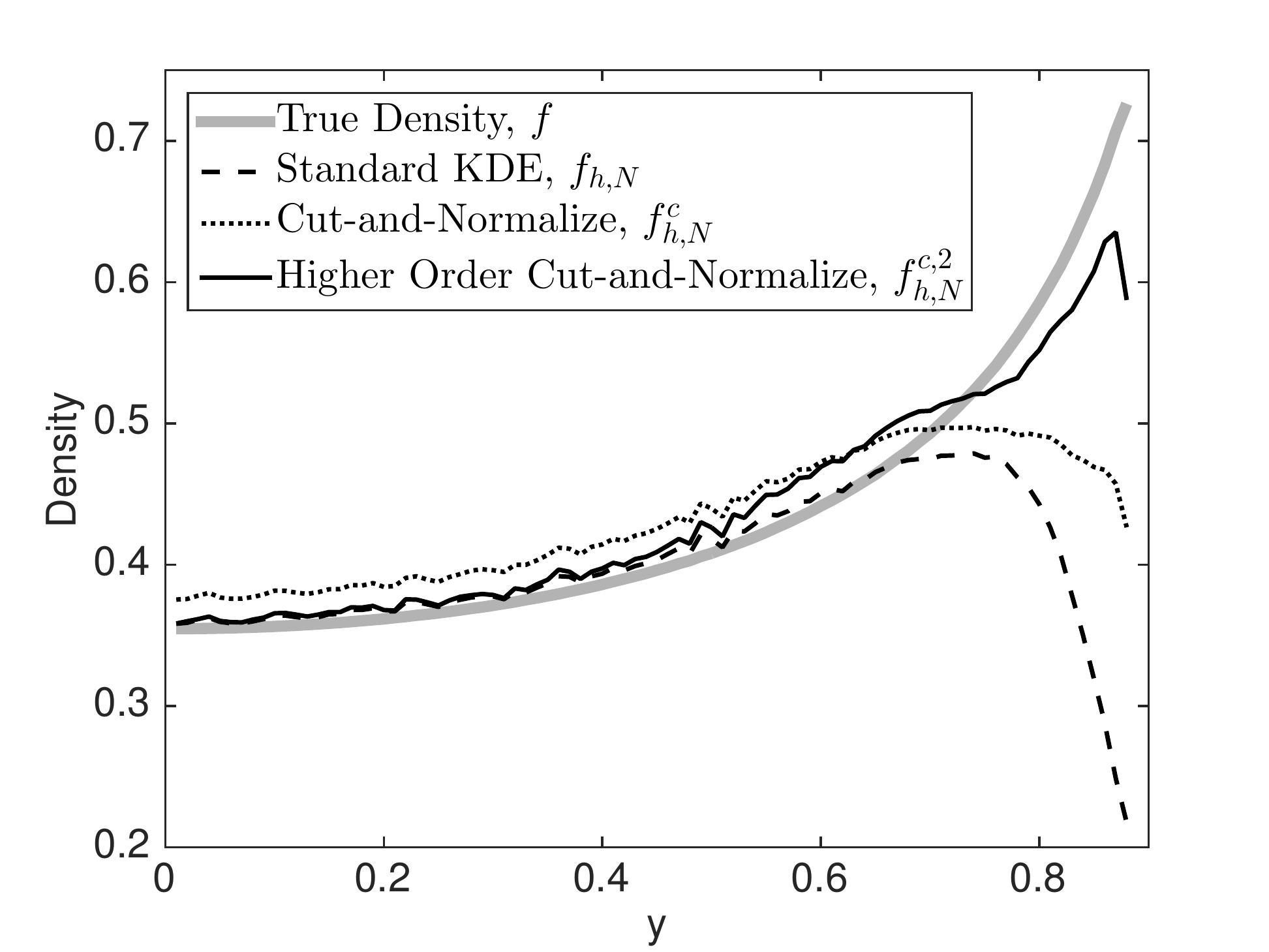}
\includegraphics[width=.96\linewidth]{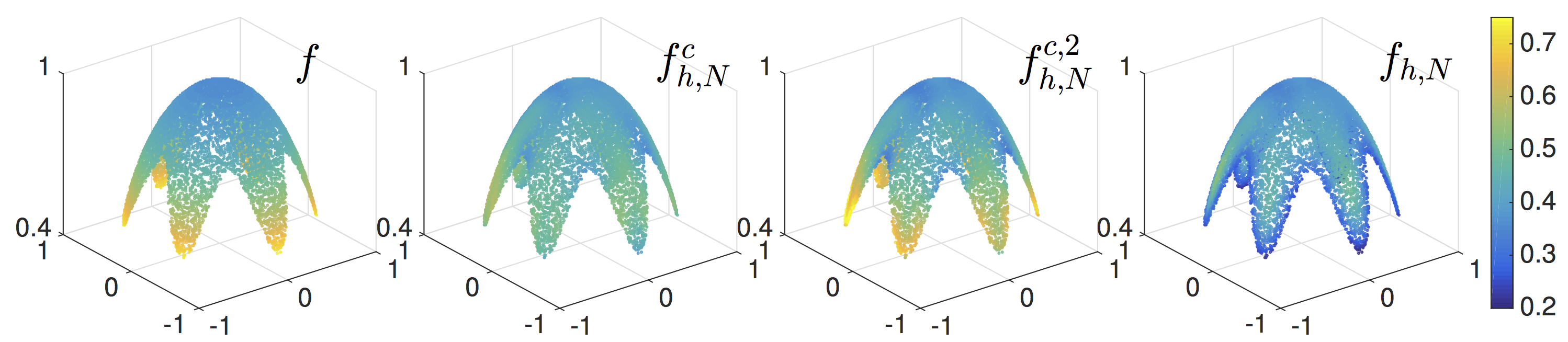}
\includegraphics[width=.98\linewidth]{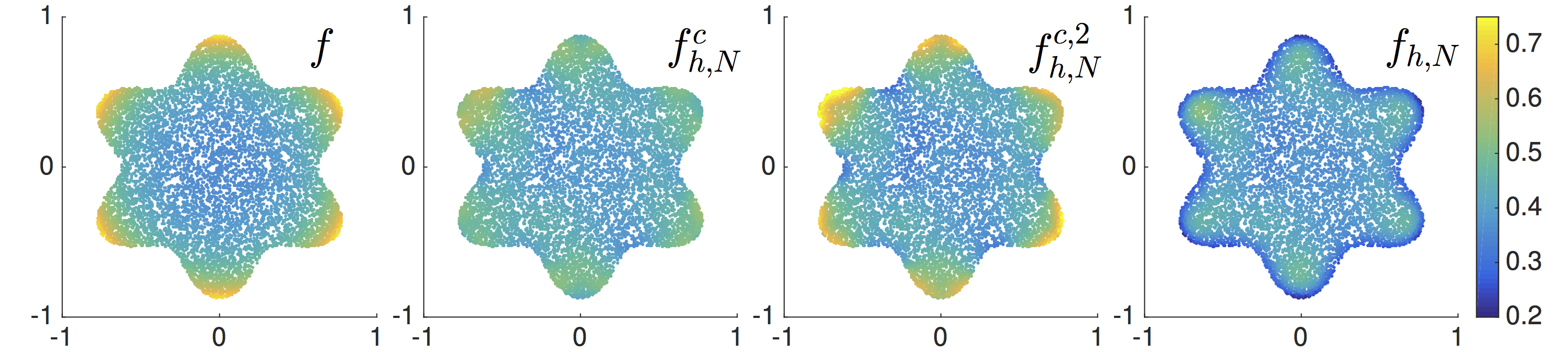}
\end{center}
\caption{\label{HOhemi} Comparison of the standard KDE, cut-and-normalize method, and higher order cut-and-normalize method on the hemisphere with oscillating boundary.  Top, Left: True density compared to estimates on the positive x-axis. Top, Right: True density compared to estimates on the positive y-axis.  Below we visualize the various estimates in 3-dimensions and 2-dimensions.}
\end{figure}

In order to make the problem more challenging we will change the sampling density to be proportional to $f(r) = (1-r^2)^{-1/2}$ which concentrates more density at the boundary.  We will create this sampling density by rejection sampling the initial density.  We first compute the normalization factor of the new density by integrating it against the volume form
\[ \alpha=\int_0^{2\pi}\int_0^{\sin(6\theta-\pi/2)/8 + 3/4} \frac{r}{1-r^2} \, dr  d\theta. \]
The new density will be $f(r) = (1-r^2)^{-1/2}/\alpha$, where $\alpha \approx 2.81893$.  In order to perform rejection sampling, note that the ratio
\[ f(r)/q(r) = \frac{73 \pi}{128\alpha(1-r^2)} \]
has maximum value $M = f(7/8)/q(7/8)$ since $r=7/8$ is the maximum radius on the oscillating boundary.  For each point sampled from $q$ a uniform random number $\xi$ is drawn; the point is accepted as a sample of $f$ if and only if $M\xi < f/q$.  After implementing this process in the  realization shown in Figure \ref{HOhemi},  the remaining  $10422$ points were independent samples of the density $f$ on the hemisphere with oscillating boundary.  Using this data set with $h=\sqrt{0.02}$, we computed the standard KDE for each data point, estimated the distance to the boundary, and computed the cut-and-normalize and higher-order cut-and-normalize estimates of the density.

The density estimates are compared visually in Figure \ref{HOhemi}.  We also repeated this experiment 10 times and computed the average of each of the estimates on the positive x-axis and the positive y-axis (which correspond to the shortest and longest radii, respectively) and these curves are compared to the true density in Figure \ref{HOhemi}.  Despite the gradient of the density increasing in the direction of the boundary, the $\mu_{h,N}$ computation still appears to have pointed into the interior as evidenced by the significant improvement of the cut-and-normalize method over the standard KDE.  This example also showed the largest difference between the cut-and-normalize method and the higher order cut-and-normalize method, possibly due to the large gradient at the boundary making the order $h$ term quite large.  The complexity of the boundary in this example illustrates the advantage of our method, which does not require any prior knowledge of the boundary.

\end{examp}

\section{Discussion}\label{discussion}

The main endeavor of this article is the generalization of the `cut-and-normalize' strategy for boundary correction \cite{gasser1979kernel,kyung1998nonparametric} to manifolds, especially when we cannot assume we know the location of the boundary.  In Section \ref{boundary} we showed that the key to extending the cut-and-normalize strategy was estimating the distance and direction of the boundary and then deriving the correct normalization factor.  Another practical consequence of this theory is that the exponential kernel has a significant advantage over other kernels, which is that the boundary normalization factor has a very simple form independent of the dimension of the manifold.  Although we illustrate a straightforward application of this strategy, it should be noted that using the distance and direction information derived here, the various boundary correction methods of \cite{schuster1985incorporating,karunamuni2008some,boundary93,boundary96,boundary99,kyung1998nonparametric,hall2002new,boundary00,boundary05,boundary14} can now be extended to manifolds and to the case where the boundary is unknown.

The new algorithm that we have described is advantageous when information about the geometric structure of the data is scarce, as is common in machine learning contexts. Blindly assuming that the manifold has no boundary can lead to serious errors near the boundary for conventional kernel density estimators. Our approach yields an algorithm that is equally accurate in the interior and the boundary, and has  performance equivalent to  the conventional approach if there is no boundary.
 Therefore, the only consideration when replacing standard KDE with the new approach is additional computational complexity. 

For simplicity, we have strongly used the ``geometric prior'' that the data lies on a manifold with unknown boundary. However, one could think of other, less constrictive, problems where the ideas could be useful. There is a range of intermediate cases where the data has ``edges'', areas where the density rolls off to zero via an outlier regime that acts as a pro forma boundary. The new algorithm may increase the estimation accuracy in these difficult regions.

\bigskip

{\bf \sf Acknowledgments.}  The research was partially supported by grants CMMI-1300007, DMS-1216568, and DMS-1250936 from the National Science Foundation.

\bibliographystyle{plain}
\bibliography{VBbib}

\appendix

\section{Proofs for Manifolds with Boundary}\label{boundaryApp}

Before proving the theorems in Section \ref{boundary}, we briefly review the assumptions we must make on the manifold $\mathcal{M}$.  Similar conditions were first introduced in \cite{heinThesis}; here we summarize the assumptions with the term \emph{tangible} which is defined for points and manifolds below. 

Recall that the exponential map $\exp_x : T_x\mathcal{M} \to \mathcal{M}$ maps a tangent vector $\vec s$ to $\gamma(||\vec s||)$ where $\gamma$ is the arclength-parametrized geodesic starting at $x$ with initial velocity $\gamma'(0) = \vec s/||\vec s||$.  The {\it injectivity radius}  $\textup{inj}(x)$ of a point $x$ is the maximum radius for which a ball in $T_x\mathcal{M}$ is mapped diffeomorphically into $\mathcal{M}$ by $\exp_x$.  In order to convert integrals over the entire manifold into integrals over the tangent space, we will use the exponential decay of the kernel to localize the integral and then change variables using the exponential map.  This requires that for a sufficiently small localization region (meaning $h$ sufficiently small) the exponential map is a diffeomorphism.  Therefore, the first requirement for kernel density estimation will be that the injectivity radius is non-zero.  

The second requirement is that the ratio
\[ R(x,y) = \frac{||x-y||}{d_I(x,y)} \]
 is bounded away from zero for $y$ sufficiently close to $x$, where $||x-y||$ is the Euclidean distance and $d_I(x,y)$ is the {\it intrinsic distance}, which is defined as the infimum of the lengths of all differentiable paths connecting $x$ and $y$.  When some path attains this infimum it is called a {\it geodesic path} and the distance is called the {\it geodesic distance} $d_g(x,y)$.  We use the intrinsic distance since it is defined for all pairs of points, whereas the geodesic distance may technically be undefined when there is no path that attains the infimum. The reason we will require $R(x,y)$ to be bounded away from zero is that the local kernel is defined in the ambient space, which makes it practical to implement. But the theory requires that the kernel decays exponentially in the geodesic distance, meaning that the kernel is localized on the manifold, not just the ambient space.  (The kernels of \cite{pelletier2005kernel,kim2013geometric} explicitly depend on the geodesic distance in order to obtain this decay.)
 
 In order to estimate the density $f$ at a point $x\in \mathcal{M}$ we require the injectivity radius $\textup{inj}(x)$ to be non-zero and the ratio $R(x,y)$ to be bounded away from zero near $x$, which motivates the following definition.
\begin{Def} We say that a point $x\in\mathcal{M} \subset \mathbb{R}^n$ is \emph{tangible} if $\textup{inj}(x)>0$ and within a sufficiently small neighborhood $N$ of $x$, $\inf_{y\in N}R(x,y) > 0$.\end{Def}
We are mainly interested in manifolds for which every point is tangible.
\begin{Def} An embedded manifold $\mathcal{M} \subset \mathbb{R}^n$ is \emph{tangible} if every $x\in \mathcal{M}$  is tangible. If there exist lower bounds for $\textup{inj}(x)$ and $\inf_{y\in\mathcal{M}}R(x,y)$ that are independent of $x$, then $\mathcal{M}$ is called \emph{uniformly tangible}. \end{Def}
 
 For example, every compact manifold as well as linear manifolds such as $\mathbb{R}^n$ are uniformly tangible. This implies that standard KDE theory on Euclidean spaces as well as existing density estimation on manifolds are included in this theory, as well as a large class of noncompact uniformly tangible manifolds.  An example where uniform tangibility fails is the 1-dimensional manifold in $\mathbb{R}^2$ given by $(r(\theta)\cos \theta,r(\theta)\sin\theta)$ where $r(\theta) = 1-1/\theta$ and $\theta \in [1,\infty)$.  Then for any $\theta \in [1,\infty)$,  set $\theta_n = \theta + 2\pi n$. The distances $d_g(\theta_n,\theta_{n+1})$ approach $2\pi$ as $n\to\infty$, whereas $||\theta_n-\theta_{n+1}||$ goes to zero.
   Thus the ratio $R(\theta_n,\theta_{n+1})$ is not uniformly bounded below on the manifold.  However, even in this example, every point on the manifold is tangible.

The key to KDE on embedded manifolds is computing the expectation,
\[ \mathbb{E}\left[\frac{1}{N h^m}\sum_{i=1}^N K\left(\frac{||x-X_i||}{h}\right)\right] = \frac{1}{h^m} \int_{\mathcal{M}} K\left(\frac{||x-y||}{h}\right)f(y)\,dV(y) \]
by splitting the integral over $\mathcal{M}$ into two disjoint regions.  Assume that $h < \textup{inj}(x)$ which implies that for some $\gamma \in (0,1)$ we have $h^{\gamma} < \textup{inj}(x)$ (we will explain the need for $h^{\gamma}$ below).  Since $h^{\gamma}$ is less than the injectivity radius, for any $s \in T_x\mathcal{M}$ with $||s|| < h^{\gamma}$ we can map $s$ to $\mathcal{M}$ diffeomorphically via $\exp_x(s) \in \mathcal{M}$.  For each $x$, we can split the manifold into the image of this ball $\exp_x(B_{h^{\gamma}}(x))$ and the complement $\mathcal{M} \cap \exp_x(B_{h^{\gamma}}(x))^c$.  In the following Lemma we show that the integral over the complement is small.  

\begin{lem}\label{lem1} Let $K:[0,\infty) \to [0,\infty)$ have exponential decay (meaning there exists constants $a_1,a_2 > 0$ such that $K(z) < a_1 e^{-a_2 z}$) and let $x \in \mathcal{M}\subset \mathbb{R}^n$ be tangible, then
\[ \frac{1}{h^m} \left| \int_{\mathcal{M} \cap \exp_x(B_{h^{\gamma}}(x))^c} K\left(\frac{||x-y||}{h}\right)f(y)\,dV(y) \right| < \mathcal{O}(h^{k}) \] 
for any $k \in \mathbb{N}$.
\end{lem}
\begin{proof}
By exponential decay, $K(z)p(z)$ is integrable for any polynomial $p$, and taking $p$ to be $z^{\ell+\kappa}$ where $\kappa$ is the degree of the polynomial upper bound of $f$ we have $K(z) < ||z||^{-\ell-\kappa}$ and therefore $|K(z)f(x+hz)| < a ||z||^{-\ell}$ for some constant $a$, where $\ell$ was arbitrary.  Making the change of variables $y=x+hz$ we find that $z \in \tilde{\mathcal{M}} \cap \exp_0(B_{h^{\gamma-1}}(0))^c$ where $\tilde{\mathcal{M}}$ is translated so that $z=0$ corresponds to the point $x \in\mathcal{M}$, and $dV(y) = h^m dV(z)$ so we have
\[ \frac{1}{h^m} \left| \int_{\mathcal{M} \cap \exp_x(B_{h^{\gamma}}(x))^c} K\left(\frac{||x-y||}{h}\right)f(y)\,dV(y) \right| \leq  \int_{\tilde{\mathcal{M}} \cap \exp_x(B_{h^{\gamma-1}}(0))^c} a ||z||^{-\ell} \,dV(z). \]
Notice that the decay of the kernel is in the ambient space distance $||z||$, whereas the region $\exp_x(B_{h^{\gamma}}(0))^c$ only guarantees that the geodesic distance from $0$ to $z$ is large.  In order for this integral to be small, we now need the guarantee that large geodesic distance implies large Euclidean distance, which is exactly our assumption that $R(x,y) > 0$.  Since $x$ is tangible, let $R(x,y) > c$, we then have, $||z|| > c d_g(0,z) > h^{\gamma-1}$, so
\[ \int_{\tilde{\mathcal{M}} \cap \exp_x(B_{h^{\gamma-1}}(0))^c} a ||z||^{-\ell} \,dV(z) \leq a c^{-\ell} \int_{\tilde{\mathcal{M}} \cap ||z||>h^{\gamma-1}} ||z||^{-\ell} \,dV(z). \]
We can bound the previous integral by the integral over all $||z||>h^{\gamma-1}$ in $\mathbb{R}^n$, and switching to polar coordinates we find
\begin{align} a c^{-\ell} \int_{\tilde{\mathcal{M}} \cap ||z||>h^{\gamma-1}} ||z||^{-\ell} \,dV(z) &= a V_n c^{-\ell} c^{-\ell} \int_{h^{\gamma-1}}^{\infty} r^{-\ell} r^n \, dr \nonumber \\ &\leq \frac{a V_n}{\ell-n-1} c^{-\ell} h^{(\gamma-1)(-\ell+n+1)} \nonumber \end{align}
for $\ell > n+2$ where $V_n$ is the volume of the unit $n$-ball. Since $\ell$ was arbitrary and $\gamma-1 < 0$, we can bound this integral by $\mathcal{O}(h^{k})$ for any $k$.  
\end{proof}

In order to extend the definition of a tangible manifold to include manifolds with boundary, we consider the tangent space for points on the boundary to be the half space.  In particular, the injectivity radius is the radius of the largest ball such that the exponential map is well defined on the intersection of the ball and the half space.  Similarly for points near the boundary, we consider the tangent space to be a cut space which is cut at $\dM_x$ in the direction $\eta_x$.  These definitions allow points on or near the boundary to still have large injectivity radii.

\begin{proof}[Theorem \ref{thm2}] Lemma \ref{lem1} shows that the integral outside the image of the ball $B_{h^{\gamma}}(x)$ is small to an arbitrarily high order in $h$. We next consider the integral inside the ball
\[ \frac{1}{h^m} \int_{\exp_x(B_{h^{\gamma}}(x))} K\left(\frac{||x-y||}{h}\right)f(y)\,dV(y). \]
Since $h^{\gamma}$ is less than the injectivity radius, we can write the integral in terms of geodesic normal coordinates $s = \exp_x^{-1}(y)$ based at $x$.  In these coordinates we have an expansion of the volume form \cite{heinThesis}
\[ dV(y) = \sqrt{|g(s)|} \, ds = \left(1-\frac{1}{6}\sum_{i,j} R_{ij} s_i s_j + \mathcal{O}(s^3) \right) \, ds \]
where $R_{ij} = \sum_{k} R_{ikjk}$ is the Ricci curvature.  Let $y = \exp_x(s)$ and let $\gamma$ be geodesic curve with $\gamma(0)=x$ and $\gamma(||s||)=y$ parametrized by arclength so that $||\gamma'(t)|| = 1$ for all $t$.  We can expand $\gamma$ in $s$ as
\[ y= \gamma(||s||) = x + s + \text{II}(s,s)/2 + \mathcal{O}(s^3) \]
where $\text{II}(s,s)$ is the second fundamental form which is bilinear in $s$ and perpendicular to $s$.  We also expand the kernel $K(||x-\exp_x(s)||/h) $ centered around $s$ as
\[ K\left(\frac{||x-\exp_x(s)||}{h}\right) = K\left(\frac{||s|| + ||\text{II}(s,s)/2|| + \mathcal{O}(s^3)}{h}\right)= K\left(\frac{||s||}{h}\right) + K'\left(\frac{||s||}{h}\right)\frac{\text{II}(s,s)}{2} + \mathcal{O}(s^3). \]
Finally, we expand the density $\tilde f(s) = f(\exp_x(s))$ around $s=0$ as
\[ f(\exp_x(s)) = f(x) + \sum_{i=1}^m s_i \frac{\partial \tilde f}{\partial s_i}(0) + \mathcal{O}(s^2). \]
We can now derive the new normalization factor
\[ m^{\partial}_0(x) = \int_{\mathbb{R}^{m-1}}\int_{-\infty}^{\dM_x/h} K(||z_{\perp}|| + |z_{\parallel}|) \, dz_{\parallel} dz_{\perp}. \]
To understand this formula, let $x^*$ be a point on the boundary which minimizes the geodesic distance, $d(x,x^*) = \dM_x$ (since a boundary is a closed set, such a point exists although it may not be unique).  If $||x-x^*|| > h$ then the boundary is far enough away that it will have a negligible effect on $m_0$ as shown in Lemma \ref{lem1}.  Thus, we restrict our attention to points with $\dM_x< h$ and we assume the $h$ is sufficiently small that $x^*$ is unique (notice that this will depend on the curvature of the boundary).  We define $\eta_x \in T_x\mathcal{M}$ to be the unit vector which points towards $x^*$, meaning that $\exp_x(\dM_x\eta_x)=x^*$ and if $x$ lies exactly on the boundary we define $\eta_x$ to be the outward pointing unit normal vector.  

 \begin{figure}[h]
\begin{center}
\includegraphics[width=.7\linewidth]{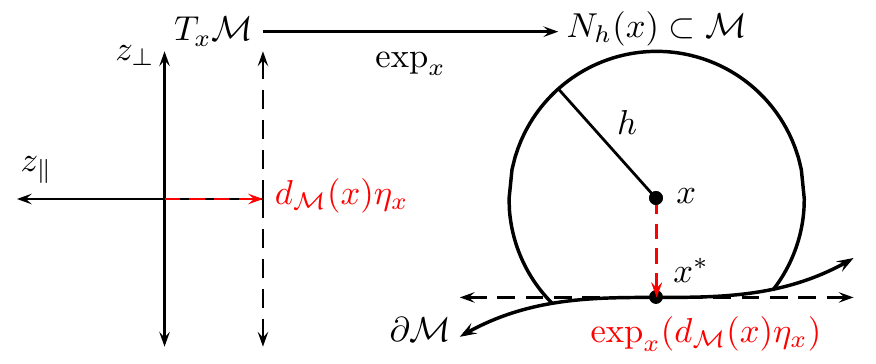}
\end{center}
\caption{\label{tangentcoords} Visualizing the tangent space near the boundary.}
\end{figure}

Next, we decompose the exponential coordinates in the tangent space $B_{h^{\gamma}}(x) \subset T_x\mathcal{M}$ into vectors $s_{\parallel}\eta_x$ (where $s_{\parallel}$ is a scalar) which are parallel to $\eta_x$ and vectors $s_{\perp}$ which are perpendicular to $\eta_x$.  All vectors perpendicular to $\eta_x$ can extend up to length $h^{\gamma}$, whereas vectors parallel to $\eta_x$ can extend up to length $h^{\gamma}$ in the direction $-\eta_x$ (away from the boundary), but only up to length $\dM_x$ in the direction $\eta_x$ (towards the boundary).  With this decomposition, 
the coefficient of $f(x)$ in the expansion is
\[  m_0^{\partial}(x) = h^{-m}\int_{[-h^{\gamma},h^{\gamma}]^{m-1}} \int_{-h^{\gamma}}^{\dM_x} K\left(\frac{||s_{\perp}|| + |s_{\parallel}|}{h}\right) \, ds_{\parallel} ds_{\perp} \]
and this is the leading order term.  Making the change of variables $s = h z$, and recalling from Lemma \ref{lem1} that the integral is negligible beyond $h^{\gamma-1}$, we can extend the integral over $z_{\perp}$ to all of $\mathbb{R}^{m-1} \subset T_x\mathcal{M}$.  On the other hand, the integral over $z_{\parallel}$ cannot be extended to all of $\mathbb{R}\subset T_x\mathcal{M}$, but only to the half-line $(-\infty,\dM_x/h]$ so that the zeroth moment becomes
\[ m_0^{\partial}(x) = \int_{\mathbb{R}^{m-1}} \int_{-\infty}^{\dM_x/h} K(||z_{\perp}|| + |z_{\parallel}|) \, dz_{\parallel} dz_{\perp}. \]
Since $m_0^{\partial}(x)$ is the coefficient of $f(x)$ in the expansion of the standard KDE formula, replacing $m_0(x)$ with $m_0^{\partial}(x)$ in the standard KDE formula yields $f_{h,N}^{\partial}(x)$ which is a consistent estimator  of $f(x)$.  

In order to establish the bias of this estimator, notice that the next term of the expansion is
\[ \sum_{i=1}^m \frac{\partial \tilde f}{\partial s_i}  h^{-m} \int_{||s||<h^{\gamma}} K\left(\frac{||s||}{h}\right)  s_i \, ds \]
which integrates to zero for $x$ sufficiently far from the boundary due to the symmetry of the domain of integration.  However, for points near the boundary, this integral will not be zero.  Instead, this term integrates to zero for every $s \perp \eta_x$ since the domain is symmetric in those directions, so we have $s = s_{\parallel}\eta_x$ and the integral becomes
\[  \eta_x \cdot \nabla f(x) h^{-m} \int_{-[h^{\gamma},h^{\gamma}]^{m-1}} \int_{-h^{\gamma}}^{\dM_x} K\left(\frac{||s_{\perp}||+|s_{\parallel}|}{h}\right) s_{\parallel}  \, ds_{\parallel} ds_{\perp}. \]
Notice that we have rewritten the partial derivatives with respect to the geodesic normal coordinates in terms of the gradient operator by inserting the metric $g_{ij}$ (which becomes the dot product) and the inverse metric $g^{jk}$ (which joins with the partial derivatives to become the gradient operator), namely
\[ \sum_{i=1}^m (\eta_x)_i \frac{\partial \tilde f(0)}{\partial s_i} = \sum_{i,j,k} (\eta_x)_i g_{ij} g^{jk}\frac{\partial \tilde f(0)}{\partial s_k} =  \sum_{i,j} (\eta_x)_i g_{ij} (\nabla f(x))_j = \eta_x \cdot \nabla f(x).\]
Changing variables to $s = hz$ as above, we find the bias to be $m_1^{\partial}(x) \eta_x \cdot \nabla f(x)$ where
\[ m_1^{\partial}(x) =  \int_{\mathbb{R}^{m-1}} \int_{-\infty}^{\dM_x/h} K(||z_{\perp}|| + |z_{\parallel}|) z_{\parallel} \, dz_{\parallel} dz_{\perp}\]
Finally, by independence of $X_i$ we have
\begin{align} \mathbb{E}& \left[ \left(\frac{1}{N h^m}\sum_{i=1}^N K\left(\frac{||x-X_i||}{h}\right) - f(x)\right)^2 \right] = \frac{1}{N^2 h^{2m}}\mathbb{E}\left[ \left(\sum_{i=1}^N (K\left(\frac{||x-X_i||}{h}\right) - h^m f(x))\right)^2 \right] \nonumber \\ &= \frac{1}{N h^{2m}}\mathbb{E}\left[ (K\left(\frac{||x-X_i||}{h}\right)- h^m f(x))^2 \right] \nonumber \\ &= \frac{1}{N}\mathbb{E}\left[ h^{-2m}K\left(\frac{||x-X_i||}{h}\right)^2 - 2h^{-m}K\left(\frac{||x-X_i||}{h}\right) f(x) +  f(x)^2 \right] \nonumber \\ &= \frac{h^{-m}}{N}\mathbb{E}\left[ h^{-m}K\left(\frac{||x-X_i||}{h}\right)^2\right] -  f(x)^2/N + \mathcal{O}(h^2/N)  \nonumber  \end{align}
which verifies the variance formula
\[ m_0^{2,\partial}(x) = \int_{\mathbb{R}^{m-1}}\int_{-\infty}^{\dM_x/h} K\left(||z_{\perp}|| + |z_{\parallel}|\right)^2 \, dz_{\parallel} dz_{\perp}. \]
\end{proof}

Next, using the asymptotic expansions of the previous proof we can easily prove Theorem \ref{thm3}.

\begin{proof}[Theorem \ref{thm3}]  The definition
\[ \mu_{h,N}(x) \equiv \frac{1}{N h^{m+1}}\sum_{i=1}^N K\left(\frac{||x-X_i ||}{h}\right)(X_i-x) \]
implies a  formula for the expectation:
\[ \mathbb{E}[\mu_{h,N}(x)] = \frac{1}{h^{m+1}}\int_{\mathcal{M}}K\left(\frac{||x-y ||}{h}\right)(y-x)f(y)\, dV(y). \]
Following Lemma \ref{lem1} we can restrict this integral to the image of the ball $||s||<h^{\gamma}$ under the exponential map, and then change variables to the geodesic normal coordinates $s\in T_x\mathcal{M}$ with $y=\exp_x(s)$, which yields
\[  \mathbb{E}[\mu_{h,N}(x)] = \frac{1}{h^{m+1}}\int_{||s||<h^{\gamma}}K\left(\frac{||x - \exp_x(s)||}{h}\right)(\exp_x(s)-x)f(\exp_x(s)) \, dV(\exp_x(s)). \]
Applying the asymptotic expansions from the proof of Theorem \ref{thm2}, we find
\begin{align}  \mathbb{E}[\mu_{h,N}(x)] &= \frac{1}{h^{m+1}}\int_{||s||<h^{\gamma}} K\left(\frac{||s||}{h}\right)f(x) s \nonumber \\ &\hspace{10pt}+ K\left(\frac{||s||}{h}\right)s\cdot \nabla f(x) s + K'\left(\frac{||s||}{h}\right) f(x) \text{II}(s,s)/2 + \mathcal{O}\left(s_i^3 K\left(\frac{||s||}{h}\right)\right) \, ds. \nonumber \end{align}
Following the proof of Theorem \ref{thm2} we decompose $s = s_{\perp} \oplus s_{\parallel}\eta_x$ and note that the first term of the integral is zero in every direction except $s = s_{\parallel}\eta_x$ which leads to
\begin{align}  \mathbb{E}[\mu_{h,N}(x)] &= \frac{1}{h^{m+1}}\int_{[-h^{\gamma},h^{\gamma}]^{m-1}} \int_{-h^{\gamma}}^{\dM_x} K\left(\frac{||s_{\perp}|| + |s_{\parallel}|}{h}\right) f(x) s_{\parallel}\eta_x \nonumber \\ &\hspace{10pt}+ K\left(\frac{||s||}{h}\right) s\cdot \nabla f(x) s + f(x) K'\left(\frac{||s||}{h}\right) \text{II}(s,s)/2 + \mathcal{O}\left(s_i^3 K\left(\frac{||s||}{h}\right)\right) \, ds_{\parallel} ds_{\perp}. \nonumber \end{align}
Changing variables to $s = hz$ we have
\begin{align}  \mathbb{E}[\mu_{h,N}(x)] &= \int_{[-h^{\gamma-1},h^{\gamma-1}]^{m-1}} \int_{-h^{\gamma-1}}^{\dM_x/h} K(||z_{\perp}|| +  |z_{\parallel}|)f(x) z_{\parallel} \eta_x \nonumber \\ &\hspace{10pt}+ h K(||z||)  z\cdot \nabla f(x) z+ f(x) K'\left(||z||\right) \text{II}(z,z)/2) + \mathcal{O}\left(h^2 z_i^3 K(||z||)\right) \,  dz_{\parallel} dz_{\perp} \nonumber \end{align}
and extending the integrals to $\mathbb{R}^{m-1}$ and $(-\infty,\dM_x/h)$ respectively (following Theorem \ref{thm2}) yields
\begin{align} \mathbb{E}[\mu_{h,N}(x)] &= \eta_x f(x) \int_{\mathbb{R}^{m-1}} \int_{-\infty}^{\dM_x/h} K(||z_{\perp}|| + |z_{\parallel}|) z_{\parallel}  \,  dz_{\parallel} dz_{\perp} + \mathcal{O}(h \nabla f(x), h f(x)) \nonumber \\
&= -\eta_x f(x) m_1^{\partial}(x) + \mathcal{O}(h,\nabla f(x), hf(x)) \nonumber
\end{align}
where we recall that the definition of the integral $m_1^{\partial}(x)$ incorporates a minus sign.

\end{proof}

\section{Dimension estimation}

Notice that the definition of $K_A$ requires the intrinsic dimension $m$ of the manifold.  Interestingly, the dimension is not required in \cite{LK} to find the Laplace-Beltrami operator of the intrinsic geometry, and in \cite{LK} the factor $\pi^{-m/2}$ is not included in the definition of a prototypical kernel.  However, in order to find a properly normalized density one must know the intrinsic dimension, and so in this paper we include the normalization factor $\pi^{-m/2}$ in the definition of the kernel for convenience.  There are many methods of identifying the dimension from the data, we advocate a method which was introduced in \cite{coifman2008TuningEpsilon} and further refined in \cite{BH14VB,IDM} which simultaneously determines the dimension and tunes the bandwidth parameter $h$.  The method of \cite{BH14VB} uses the fact that when $h$ is well tuned, the unnormalized kernel sum $\frac{1}{N}\sum_{i=1}^N K\left(\frac{||x-x_i||}{h}\right)$ is proportional to $h^m$ as in Theorem \ref{cor2}.  By varying $h$ one can estimate the scaling law $m = \frac{d\log D(h)}{d\log h}$, and when $h$ is well tuned this scaling law will be stable under small changes in $h$.  

In order to simultaneously estimate the dimension $m$ and tune the bandwidth $h$, we first generate a grid of $h$ values, $h_j$ (typically a logarithmic scale is used, such as $h_j = 1.1^{j}$ for $j=-20,...,0,...,20$).  We then evaluate the sum
\[ S(x,h_j) = \frac{1}{N}\sum_{i=1}^N K\left(\frac{||x-x_i||}{h_j}\right) \]
which should be proportional to $h^{m}$ when $h=h_j$ is well tuned.  Motivated by this, we compute the scaling law at each $h_j$ by
\[ \textup{dim}(x,h_j) = \frac{\log(S(x,h_{j+1})) - \log(S(x,h_j)}{\log(h_{j+1}) - \log(h_{j})} \approx  \frac{d\log S}{d\log h}(h_j) \]
which gives us an approximate dimension for each value of $h_j$.  In \cite{BH14VB} they advocated taking value of $h_j$ which maximizes the dimension, however in \cite{IDM} they showed that the extrinsic curvature can lead to overestimation.  Instead, \cite{IDM} advocates looking for persistent values of dimension, which intuitively means one should look for values of the dimension such that the curve $\textup{dim}(x,h_j)$ is flat for a large range of values of $h_j$.  One method is to approximate derivatives of $\textup{dim}(x,h_j)$ with respect to $h_j$ and attempt maximize $\textup{dim}$ while minimizing the derivatives.  

Notice that the above method finds a dimension at a single point $x$.  To estimate a single dimension for an entire data set, one can define $S(h_j)$ to be the average value of $S$ over the entire data set and apply the same procedure.

\end{document}